\documentclass{amsart}[12pt]

\usepackage{hyperref}
\usepackage{etoolbox}
\usepackage{mathrsfs}
\usepackage{graphicx}
\usepackage{color,latexsym,amsfonts,amsthm,amsmath,amssymb,path,enumitem}
\usepackage{tikz}
\usepackage{algpseudocode}
\usepackage{algorithm}

\usetikzlibrary{decorations.pathreplacing}
\AtBeginEnvironment{quote}{\smaller}

\newtheorem{thm}{Theorem}[section]
\newtheorem{lem}[thm]{Lemma}

\newtheorem{cor}[thm]{Corollary}

\newtheorem{Defn}[thm]{Definition}

\newtheorem{defin}[thm]{Definition}

\usepackage{dsfont}

\def\.{\hskip.06cm}

\def\.{\hskip.06cm}

\def\<{\langle}
\def\>{\rangle}

\usepackage{filecontents}

\begin{document}

\title[Tiling with small tiles]{Tiling with small tiles}

\author[Anne~Kenyon]{ \ Anne~Kenyon$^\star$}
\author[Martin~Tassy ]{ \ Martin~Tassy$^\dagger$}

\begin{abstract}
We look at sets of tiles that can tile any region of size greater than 1 on the square grid. This is not the typical tiling question, but relates closely to it and therefore can help solve other tiling problems -- we give an example of this. We also present a result to a more classic tiling question with dominoes and L-shape tiles. 
\end{abstract}

\thanks{\thinspace ${\hspace{-.45ex}}^\dagger$Department of Mathematics,
UCLA, Los Angeles, CA, 90095.
\hskip.06cm
Email:
\hskip.06cm
\texttt{\{mtassy\}@math.ucla.edu}}

\thanks{\thinspace ${\hspace{-.001ex}}^\star$Weizmann Institute Faculty of Mathematics and Computer Science, Rehovot, Israel. \hskip.06cm
Email:
\hskip.06cm
\texttt{anne.kenyon@weizmann.ac.il}.}

\maketitle

\section{Introduction}
Consider a set of tiles $T$ and a connected region $R$ in the lattice $\mathbb{Z}^2$. Deciding if  $R$ is tileable with copies of the tiles in $T$ is a classical \emph{tileability problem}, occupying a major role in Combinatorics and Discrete Geometry. For a generic set of tiles, the answer to the tileability problem is hard to provide, and even the easier problem of deciding the tileability of simply connected regions has been shown to be NP-complete for numerous models (see amongst others ~\cite{Moore2001hard}, ~\cite{PakYang2013} and ~\cite{PakYang2013bis}).

 As a consequence, most research related to tileability problems focuses on small sets of small tiles. That is, sets of tiles containing no more than two tiles of small size, including or not rotations of those tiles. However, even in this setting, the complexity of the answers to the tileability problem are very diverse. Some sets of tiles exemplifying different behaviours include: the set of dominoes, for which the tileability can be decided in polynomial time for general regions (see amongst many others ~\cite{Kasteleyn1987},~\cite{fournier1996pavage},~\cite{Thiant2003}), the set of two rectangles, for which the tileability is decidable in polynomial time for simply connected regions but is NP-complete in the general case (see~\cite{beauquier1995tiling},~\cite{kenyon1992tiling}), and also the set of $L$-shape tiles and their rotations, for which the tileability is NP-complete even for simply connected regions (see~\cite{Moore2001hard}). Our paper is inscribed in this context of tiling with small tiles.
 
We focus on two problems which we believe are specifically interesting among sets of small tiles. This first one concerns tiling with domino and L-shape tiles as represented in Figure \ref{fig:tile}. This set of tiles has a strategical role amongst small tiles problems  because it is the minimal set for which the \emph{Conway group} of the tiling is trivial. In other words, this is the simplest set of tiles $T$ such that any region of $\mathbb{Z}^{2}$ is connected to the empty set by a sequence of adding and removing tiles in $T$. In particular, the height functions method developed by Conway and Lagarias in~\cite{conway1990tiling} and which has proved to be very useful for numerous models (dominos~\cite{fournier1996pavage}, bars~\cite{kenyon1992tiling} ~\cite{beauquier1995tiling}, and ribbon tilings~\cite{sheffield2002ribbon}) cannot work in this context. However, suprisingly, we show that the complexity of the tileability problem for general polyominoes is quasi-linear in this setting and we obtain the following theorem:
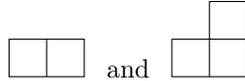
\begin{figure}[h]
 \label{fig:tile}
\begin{center}
\begin{tikzpicture}[scale = 0.5]
\draw (0,1)--(0,0)--(2,0)--(2,1)--(0,1);
\draw (1,0)--(1,1);
\end{tikzpicture}
\, and\,\,\,
\begin{tikzpicture}[scale=0.5]
\draw (3,1)--(3,0)--(5,0)--(5,1)--(4,1)--(4,2)--(5,2)--(5,1);
\draw (4,0)--(4,1)--(3,1);
\end{tikzpicture}
\end{center}
\caption{Set $T$, consisting of the domino (left) and L-tile (right)}
\end{figure}

\begin{thm}
\label{thm:twotiles}
Let $R$ be a connected polyomino in the plane $\mathbb{Z}^{2}$, and let $n=|R|$ be the size of $R$. There exists an algorithm that decides in time $O(n\log n)$ whether $R$ can be tiled with the two tiles of Figure~\ref{fig:tile}.
\end{thm}

Moreover, like in the domino and ribbon tiles cases, we show that in this setting there exist two local moves that connect any two tilings of a given region. 

The second problem we study in this paper considers, in a sense, the opposite approach. Which sets of tiles can tile \emph{any} region? We show how to find these sets in any dimension. More precisely and among other things, we prove the following result:
\begin{thm}\label{thm:ndimensions}
There exists a set of $\frac{d(d+3)}{2}$ tiles that can tile any polyomino in dimension $d \geq 2$, excluding the unit cube. There exists a linear-time algorithm to find a tiling of the polyomino with these tiles. 
\end{thm}

 The rest of the paper is structured as follows. In Section 2 we prove Theorem \ref{thm:twotiles} by describing an algorithm which decides the tileability of a polyomino by the two different tiles represented in Figure~\ref{fig:tile}.  We also show that the complexity of this algorithm is $O(n\log n)$. In Section 3 we prove Theorem \ref{thm:ndimensions} and study a few questions related to sets that can tile any polyomino.  Lastly, we conclude with a Future Research section. 
 
Before diving into the results, we define the setting. A \emph{polyomino} is a connected region with vertices in $\mathbb{Z}^d$, and edges of length 1. A polyomino is not necessarily simply connected and might contain holes. The \emph{size} of a polyomino is the number of unit $d$-dimensional cubes it contains.  A \emph{tile} is a simply connected polyomino with a fixed orientation. A polyomino is \emph{tileable} with a given set of tiles if it can be entirely covered with any number of copies of the tiles in the set, without overlap. We say that set $S$ \emph{tiles} polyomino $R$ if $R$ is tileable by $S$. A \emph{tiling} of $R$ by $S$ is an instantiation of a solution to the tiling problem: a presentation of copies of tiles from $S$ in their locations in $R$ such that they cover $R$. Giving a valid tiling is one way to prove that $R$ is tileable by $S$.


\section{Tiling with the Domino and L-Shape} 

We will now prove Theorem~\ref{thm:twotiles}, resolving the tiling question for set $T$. The algorithm, which we call ALG, is a greedy constructive algorithm that takes as input the polyomino $R$ and outputs a specific tiling of $R$ if the region $R$ is tileable, and an error otherwise. We begin by introducing a few more definitions and clarifications that will be used for the proof. 

Recall that both the polyomino $R$ and the tiles in $T$ are oriented and cannot be rotated. Every square in $R$ on $\mathbb{Z}^2$ has an integral $x$ and $y$ coordinate. Two squares are \emph{adjacent} if they share a side. Without loss of generality let the lowest square in $R$ have $y$-coordinate 1. The \emph{height} of a square is its $y$-coordinate.

A \emph{row} in $R$ is a set of adjacent squares that all have the same $y$-coordinate, and has no square in $R$ adjacent to the left or right of it. The \emph{length} of a row is the number of squares it contains. A \emph{top row} in $R$ is a row with no squares in $R$ adjacent above it. The \emph{leftmost top row} is the top row with the leftmost right end square (with the smallest maximum $x$-coordinate). If there are multiple such top rows, define the highest one to be the leftmost top row (the one with the largest $y$-coordinate). Lastly, the definition of a \emph{top left corner} in $R$ includes two cases: firstly, a top left corner can be a region of at least 2 squares adjacent horizontally that have no squares in $R$ adjacent above them, and the left square has no square in $R$ adjacent below or to the left of it. Secondly, a top left corner can be of vertical and horizontal width two, with no square of $R$ adjacent to the left or above it. It may help in understanding the definition to realize that top left corners require at least one domino to be tileable. Depicted here: 
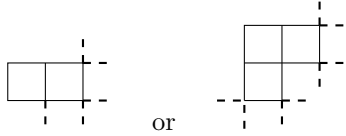
\begin{figure}[h]
\caption{both cases of a \emph{top left corner}}
\begin{center}
\begin{tikzpicture}[scale=0.5]
 \draw (2,0)--(0,0)--(0,1)--(2,1)--(2,0);
 \draw (1,0)--(1,1);
 \draw[thick,dashed] (1,0)--(1,-0.75);
 \draw[thick,dashed] (2,1)--(2,1.75);
 \draw[thick,dashed] (2,1)--(2.75,1);
 \draw[thick,dashed] (2,0)--(2.75,0);
 \draw[thick,dashed] (2,0)--(2,-0.75);
\end{tikzpicture}
 \;\;\; or \;\;\;
 \begin{tikzpicture}[scale = 0.5]
\draw (0,0)--(0,2)--(2,2);
\draw (1,0)--(1,2);
\draw (0,1)--(2,1);
\draw (0,0)--(1,0);
\draw (2,2)--(2,1);
\draw[thick,dashed] (0,-0.75)--(0,0)--(-0.75,0);
\draw[thick,dashed] (2,2)--(2.75,2);
\draw[thick,dashed] (2,2)--(2,2.75);
\draw[thick,dashed] (1,0)--(1.75,0);
\draw[thick,dashed] (1,0)--(1,-0.75);
\draw[thick,dashed] (2,1)--(2.75,1);
\draw[thick,dashed] (2,1)--(2,0.25);
\end{tikzpicture}
\end{center}
\label{topleftcorner}
\end{figure}

\subsection*{Description of the Algorithm}
ALG greedily places tiles in $R$ until the polyomino is fully covered, in which case it returns the tiling, or until the tiles cannot fit, in which case it returns an error. The exact process is described below in Algorithm 1. In this implementation, ALG updates the region each time it places a tile, removing the area covered by that tile from the region. The variable $R_c$ refers to the current subregion of $R$ that remains to be tiled.

Essentially, ALG fills all top rows of even length and top left corners with dominoes, and then, once there are no more even length top rows or top left corners, ALG finds the leftmost top row and attempts to place an L-tile in the leftmost position on that row. If the L-tile does not fit there, ALG returns that the polyomino $R$ is not tileable. 

\begin{algorithm}
\caption{ALG}
\begin{algorithmic}
\State $R_c = R$
\While{$|R_c| > 0$}
    \If {(A:) there is a top row of even length in $R_c$}
        \State tile the entire row with dominoes.
    \ElsIf {(B:) there is a top left corner in $R_c$}
            \State tile with one domino in the corner.
    \Else
        \State Find the leftmost top row $r$ in $R_c$.
        \If {(C:) an $L$-tile can be placed in the leftmost position of $r$}
            \State tile the leftmost position of $r$ with an $L$-tile.
        \Else
            \State output error ``$R$ is not tileable".
        \EndIf 
    \EndIf
    \State $R_c \leftarrow R_c $ without newly tiled subregion.
\EndWhile
\State output the tiling of $R$.
\end{algorithmic}
\end{algorithm}

\subsection*{Correctness of the algorithm}

We want to prove that ALG returns a valid tiling when one exists, and returns an error if and only if no valid tiling exists. \\\\
To this end, we first define in Figures~\ref{move1} and~\ref{move2} two local moves on a tiling. For both local moves, the same subregion is tiled by a different configuration of tiles. These two moves will be used in essence to prove that if a valid tiling exists, then it can be manipulated to transform it into a tiling that ALG would have found. 
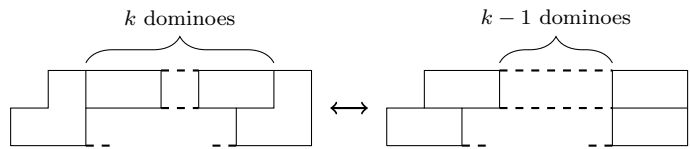
\begin{figure}[h]
\begin{tikzpicture}[scale = 0.5]
\draw (0,0) -- (2,0) -- (2,1) -- (4,1) -- (4,2) -- (1,2) -- (1,1) -- (0,1) -- (0,0);
\draw (2,1) -- (2,2);
\draw[thick,dashed] (2,0) -- (2.75,0);
\draw[thick,dashed] (4,1) -- (5,1);
\draw[thick,dashed] (4,2) -- (5,2);
\draw (6,0) -- (8,0) -- (8,2) -- (7,2) -- (7,1) -- (6,1) -- (6,0);
\draw (6,1) -- (5,1) -- (5,2) -- (7,2);
\draw[thick,dashed] (6,0) -- (5.25,0);
\draw[thick,<->] (8.5,1) -- (9.5,1);
\draw [decorate,decoration={brace,amplitude=10pt}] (2,2.2) -- (7,2.2) node [black,midway,yshift=0.6 cm] {\footnotesize $k$ dominoes};

\draw (10,0) -- (12,0) -- (12,1) -- (13,1) -- (13,2) -- (11,2) -- (11,1) -- (10,1) -- (10,0);
\draw (11,1) -- (12,1);
\draw[thick,dashed] (12,0) -- (12.75,0);
\draw[thick,dashed] (13,1) -- (16,1);
\draw[thick,dashed] (13,2) -- (16,2);
\draw (16,0) -- (18,0) -- (18,2) -- (16,2) -- (16,0);
\draw (16,1) -- (18,1);
\draw[thick,dashed] (16,0) -- (15.25,0);
\draw [decorate,decoration={brace,amplitude=10pt}] (13,2.2) -- (16,2.2) node [black,midway,yshift=0.6 cm] {\footnotesize $k-1$ dominoes};
\end{tikzpicture}
\caption{Local move 1}
\label{move1}
\end{figure}

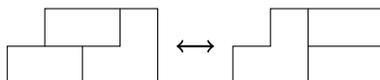
\begin{figure}[h]
\begin{tikzpicture}[scale = 0.5]
\draw (0,0) -- (4,0) -- (4,2) -- (1,2) -- (1,1) -- (0,1) -- (0,0);
\draw (1,1) -- (3,1) -- (3,2);
\draw (2,0) -- (2,1);
\draw[thick,<->] (4.5,1) -- (5.5,1);
\draw (6,0) -- (10,0) -- (10,2) -- (7,2) -- (7,1) -- (6,1) -- (6,0);
\draw (8,0) -- (8,2);
\draw (8,1) -- (10,1);
\end{tikzpicture}
\caption{Local move 2}
\label{move2}
\end{figure}

\begin{lem}\label{lem:onel}
Let $R$ be a polyomino tileable by $T$. Let $r$ be a leftmost top row of odd length with height greater than $1$ in $R$. There exists a tiling of $R$ such that exactly one L-tile is used in the tiling of $r$.
\end{lem}

\begin{proof}
Since $r$ is of odd length, there must be an odd number of L-tiles placed on $r$. 
If there exists a valid tiling of $R$ that places more than one L-tile on $r$, then using local move 1 repeatedly, there also exists a valid tiling of $R$ that replaces the L-tiles in pairs until only one remains.  
\end{proof}

\begin{lem}\label{lem:leftend}
Let $R$ be a polyomino tileable by $T$. Let $r$ be the first leftmost top row of odd length in $R$ processed by ALG (in step C of Algorithm 1). If $r$ is at height greater than 1, then there exists a tiling of $R$ such that $r$ is tiled with a unique L-tile at its left end and dominoes everywhere else.
\end{lem}

\begin{proof}  We show that if a valid tiling of $R$ exists, then it is always possible to use local moves $1$ and $2$ so that $r$ is tiled with a unique $L$-tile at its left end. The proof works by induction on the height of  the leftmost top row $r$. \\

\textbf{If $r$ is at height $h \geq 3$ (general case)}. Lemma~\ref{lem:onel} limits the cases necessary to consider to ones with exactly one L-tile, as represented by the following schematic: \\\\\\
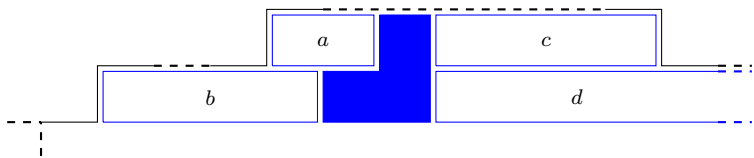
\begin{tikzpicture}[scale = 0.75]
\draw[thick,dashed] (-0.6,0) -- (0,0);
\draw[thick,dashed] (0,-0.6) -- (0,0);
\draw (0,0) -- (1,0) -- (1,1) -- (2,1);
\draw[thick,dashed] (2,1) -- (3,1);
\draw (3,1) -- (4,1) -- (4,2) -- (5,2);
\draw[thick,dashed] (5,2) -- (10,2);
\draw (10,2) -- (11,2) -- (11,1) -- (12,1);
\draw[thick,dashed] (12,1) -- (12.6,1);

\draw [blue, fill=blue] (6,1.9) -- (6.9,1.9) -- (6.9,0) -- (5,0) -- (5,0.9) -- (6,0.9) -- (6,1.9);
\draw [blue] (4.1,1) -- (4.1,1.9) -- (5.9,1.9) -- (5.9,1) -- (4.1,1) node[black,midway,yshift=0.33cm] {\footnotesize $a$};
\draw [blue] (1.1,0) -- (1.1,0.9) -- (4.9,0.9) -- (4.9,0) -- (1.1,0) node[black,midway,yshift=0.33cm] {\footnotesize $b$};
\draw [blue] (7,1) -- (7,1.9) -- (10.9,1.9) -- (10.9,1) -- (7,1) node[black,midway,yshift=0.33cm] {\footnotesize $c$};
\draw [blue] (12,0) -- (7,0) -- (7,0.9) -- (12,0.9) node[black,midway,yshift=-0.33cm] {\footnotesize $d$};
\draw[thick,dashed, blue] (12,0) -- (12.6,0);
\draw[thick,dashed, blue] (12,0.9) -- (12.6,0.9);
\end{tikzpicture}\\\\
Subregions $a$ and $c$ are of even length. Also, $b$ is at least as large as $a$, or else it would have been a top left corner and tiled by step B in Algorithm 1. We now need to show that since this tiling exists, there must exist another valid tiling with the L-tile all the way to the left end of the top row. 

Local moves 1 and 2 can be used to slide the L-tile to the left end. Since this does not affect regions $c$ and $d$, they remain validly tiled. There are now two cases to consider:
\begin{itemize}
    \item If $b$ is of even length, then local move 2 can be applied a number of times equal to the number of dominoes in $a$. This slides the L-tile to the left end of $r$, and all other regions remain validly tiled. 
    \item If $b$ is of odd length, notice that if the L-tile and region $a$ are removed to form a new polyomino $R'$, then $b$ becomes the new leftmost top row. Thus, using the induction hypothesis, $b$ is tiled with a unique L-tile at its left end. And therefore, like in the even case, local move $2$ can be applied a number of times equal to the number of dominoes in $a$, and the L-tile can be slid to the left end of $r$.\\
\end{itemize}

\textbf{If $r$ is at height 2 (base case)}. The same schematic applies as for the height $h \geq 3$ case, except that region $R$ does not extend below regions $b$ and $d$. \\\\
In this scenario, using local moves 1 and 2 does not change the parity of regions $b$ and $d$. We again have two cases to consider:
\begin{itemize}
    \item If both $b$ and $d$ are of even length, local move 2 can be applied to slide the $L$-tile to the left end of $r$, and then $b$ and $d$ can be tiled entirely with dominoes, thereby yielding a solution tiling with an $L$-tile in the leftmost position on $r$.
    \item If either $b$ or $d$ is of odd length, then $R$ cannot be tiled. A row of odd length in $R$ must contain an $L$-tile for $R$ to be tileable, but a row of odd length can only contain an $L$-tile if it is at height $>1$.\\
\end{itemize}

\end{proof}

\subsection*{Proof of Theorem~\ref{thm:twotiles}} First note that the algorithm ALG terminates, because step A tiles all top rows of even length, and step C tiles all top rows of odd length. For any non-empty $R_c$, there will be a top row according to our definitions. Therefore, ALG always makes progress and reduces $R_c$ to a smaller $R_c'$ by performing steps A and C, and thus will eventually reach a size zero $R_c$, unless the region is found to be not tileable.\\

We now prove that each step in ALG leads to a valid tiling if and only if $R$ is tileable. 
\begin{itemize}
    \item In step A, tiling an even row with dominoes is at least as good as tiling with L-tiles. If there exists a valid tiling with L-tiles, local move 1 can replace them all with dominoes.
    \item In step B, the domino is the only tile that can fill the top left square of a top left corner, therefore a solution tiling must have a domino in that position.
    \item Step C is proven valid by Lemma~\ref{lem:leftend} for cases where the odd-length leftmost top row $r$ is at height $>1$ in the region $R_c$. If $r$ is at height 1, then there is no possible tiling and the algorithm correctly returns not tileable.
    
\end{itemize}

Thus, we have proven that this algorithm indeed determines whether a polyomino $R$ can be tiled with $S$. The following Corollary is a direct consequence our proof:

\begin{cor}\label{cor:localmovetransform}
Any two tilings of the region $R$ can be transformed from one to the other by a sequence of local moves $1$ and $2$. 
\end{cor}


\section{Fountain Sets: Sets of Tiles that Tile Any Polyomino}

Which sets of tiles can tile any region? One set is obvious: the set consisting of simply the $d$-dimensional unit cube. Any set that contains the unit cube can also tile any polyomino. But what about others sets, ones that do not contain the unit cube? First of all, a set that does not contain the unit cube cannot tile a polyomino which is of size one. Therefore, if the unit cube is excluded from the tile set, it must also be excluded from the possibilities for the polyomino to be tiled. 

In this section, we present sets that can tile any region, excluding some small regions like the unit cube. We call these sets \emph{fountain sets}. We will begin by defining fountain sets and then study some of their properties, and then give an example of how they can be used to solve other tiling problems. Note that in this section, all tiles can now be rotated. 

\begin{defin}
A tile set $S$ is a \emph{fountain set} if it has the following property: A polyomino $P$ composed of a single tile in $S$ plus any adjacent unit cube is tileable by $S$. 
\end{defin}

Furthermore, we define a fountain set $S$ to be \emph{minimal} if no tile can be removed from $S$ without breaking the fountain property. Consider the set of tiles $S_2$ represented in Figure~\ref{fig:5tiles}. This is an example of a fountain set in 2-dimensions (and in fact is a minimal fountain set).

\begin{figure}[h]
\begin{tikzpicture}[scale=0.50]

\draw[thick,-] (0,1)--(0,0)--(2,0)--(2,1)--(0,1);
\draw[very thin,-] (1,0) -- (1,1);

\draw[thick,-] (3,2)--(3,0)--(5,0)--(5,1)--(4,1)--(4,2)--(3,2);
\draw[-] (4,0)--(4,1)--(3,1);

\draw[thick,-] (6,1)--(6,0)--(9,0)--(9,1)--(6,1);
\draw (7,0) -- (7,1);
\draw (8,0) -- (8,1);

\draw[thick,-] (10,1)--(10,0)--(13,0)--(13,1)--(12,1)--(12,2)--(11,2)--(11,1)--(10,1);
\draw (11,0)--(11,1)--(12,1)--(12,0);

\draw[thick,-] (15,0)--(15,1)--(14,1)--(14,2)--(15,2)--(15,3)--(16,3)--(16,2)--(17,2)--(17,1)--(16,1)--(16,0)--(15,0);
\draw (15,1)--(16,1)--(16,2)--(15,2)--(15,1);
\end{tikzpicture}
\caption{Set $S_2$: a domino, L-tile, 3-bar, T-tile, and plus-tile}
\label{fig:5tiles}
\end{figure}
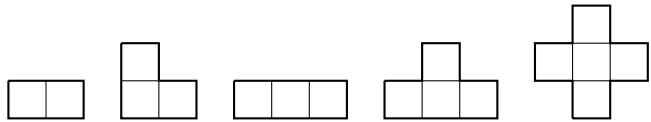

\begin{lem}\label{lemma:addsquare}
The set $S_2$ given in Figure~\ref{fig:5tiles} is a fountain set. 
\end{lem}

\begin{proof}
Lemma~\ref{lemma:addsquare} is proved by going through all the cases in Figure~\ref{fig:proof2}.  Examples that are symmetric up to rotation are not repeated.

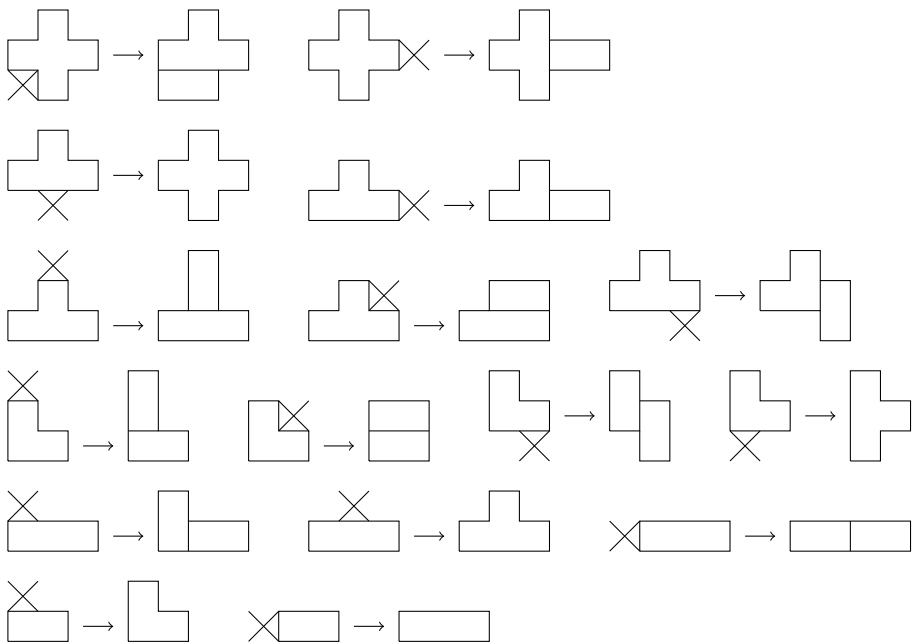
\begin{figure}[h]
\begin{tikzpicture}[scale=0.40]

\draw (0,0) -- (2,0)-- (2,1)--(0,1)--(0,0); 
\draw (0,1) -- (1,2);
\draw (1,1) -- (0,2);
\draw[->] (2.5,0.5) -- (3.5,0.5);
\draw (4,0)--(6,0)--(6,1)--(5,1)--(5,2)--(4,2)--(4,0);

\draw (9,0) -- (11,0) -- (11,1) -- (9,1) -- (9,0);
\draw (8,0) -- (9,1);
\draw (9,0) -- (8,1);
\draw[->] (11.5,0.5) -- (12.5,0.5);
\draw (13,0) -- (16,0) -- (16,1) -- (13,1) -- (13,0);

\draw (0,3) -- (3,3) -- (3,4) -- (0,4) -- (0,3);
\draw (0,4) -- (1,5);
\draw (1,4) -- (0,5);
\draw[->] (3.5,3.5) -- (4.5,3.5);
\draw (5,3) -- (8,3) -- (8,4) -- (6,4) -- (6,5) -- (5,5) -- (5,3);
\draw (6,3) -- (6,4);

\draw (10,3) -- (13,3) -- (13,4) -- (10,4) -- (10,3);
\draw (11,4) -- (12,5);
\draw (12,4) -- (11,5);
\draw[->] (13.5,3.5) -- (14.5,3.5);
\draw (15,3) -- (18,3) -- (18,4) -- (17,4) -- (17,5) -- (16,5) -- (16,4) -- (15,4) -- (15,3);

\draw (20,3) -- (21,4);
\draw (21,3) -- (20,4);
\draw (21,3) -- (24,3) -- (24,4) -- (21,4) -- (21,3);
\draw[->] (24.5,3.5) -- (25.5,3.5);
\draw (26,3) -- (30,3) -- (30,4) -- (26,4) -- (26,3);
\draw (28,3) -- (28,4);

\draw (0,6) -- (2,6) -- (2,7) -- (1,7) -- (1,8) -- (0,8) -- (0,6);
\draw (0,8) -- (1,9);
\draw (1,8) -- (0,9);
\draw[->] (2.5,6.5) -- (3.5,6.5);
\draw (4,6) -- (6,6) -- (6,7) -- (5,7) -- (5,9) -- (4,9) -- (4,6);
\draw (4,7) -- (5,7);

\draw (8,6) -- (10,6) -- (10,7) -- (9,7) -- (9,8) -- (8,8) -- (8,6);
\draw (9,7) -- (10,8);
\draw (10,7) -- (9,8);
\draw[->] (10.5,6.5) -- (11.5,6.5);
\draw (12,6) -- (14,6) -- (14,8) -- (12,8) -- (12,6);
\draw (12,7) -- (14,7);

\draw (16,7) -- (18,7) -- (18,8) -- (17,8) -- (17,9) -- (16,9) -- (16,7);
\draw (17,7) -- (18,6);
\draw (18,7) -- (17,6);
\draw[->] (18.5,7.5) -- (19.5,7.5);
\draw (20,7) -- (21,7) -- (21,9) -- (20,9) -- (20,7);
\draw (21,8) -- (22,8) -- (22,6) -- (21,6) -- (21,7);

\draw (24,7) -- (26,7) -- (26,8) -- (25,8) -- (25,9) -- (24,9) -- (24,7);
\draw (24,7) -- (25,6);
\draw (24,6) -- (25,7);
\draw[->] (26.5,7.5) -- (27.5,7.5);
\draw (28,6) -- (29,6) -- (29,7) -- (30,7) -- (30,8) -- (29,8) -- (29,9) -- (28,9) -- (28,6);

\draw (0,10) -- (3,10) -- (3,11) -- (2,11) -- (2,12) -- (1,12) -- (1,11) -- (0,11) -- (0,10);
\draw (1,12) -- (2,13);
\draw (2,12) -- (1,13);
\draw[->] (3.5,10.5) -- (4.5,10.5);
\draw (5,10) -- (8,10) -- (8,11) -- (7,11) -- (7,13) -- (6,13) -- (6,11) -- (5,11) -- (5,10);
\draw (6,11) -- (7,11);

\draw (10,10) -- (13,10) -- (13,11) -- (12,11) -- (12,12) -- (11,12) -- (11,11) -- (10,11) -- (10,10);
\draw (13,11) -- (12,12);
\draw (12,11) -- (13,12);
\draw[->] (13.5,10.5) -- (14.5,10.5);
\draw (15,10) -- (18,10) -- (18,12) -- (16,12) -- (16,11) -- (15,11) -- (15,10);
\draw (16,11) -- (18,11);

\draw (20, 11) -- (23,11) -- (23,12) -- (22,12) -- (22,13) -- (21,13) -- (21,12) -- (20,12) -- (20,11);
\draw (22,11) -- (23,10);
\draw (23,11) -- (22,10);
\draw[->] (23.5,11.5)--(24.5,11.5);
\draw (25, 11) -- (27,11) -- (27,10) -- (28,10) -- (28,12) -- (27,12) -- (27,13) -- (26,13) -- (26,12) -- (25,12) -- (25,11);
\draw (27,11) -- (27,12);

\draw (0,15) -- (3,15) -- (3,16) -- (2,16) -- (2,17) -- (1,17) -- (1,16) -- (0,16) -- (0,15);
\draw (1,15) -- (2,14);
\draw (2,15) -- (1,14);
\draw[->] (3.5, 15.5) -- (4.5,15.5);
\draw (5, 15) -- (6,15) -- (6, 14) -- (7,14) -- (7,15) -- (8,15) -- (8,16) -- (7,16) -- (7,17) -- (6,17) -- (6,16) -- (5, 16) -- (5, 15);

\draw (10,14) -- (13, 14) -- (13, 15) -- (12,15) -- (12, 16) -- (11, 16) -- (11, 15) -- (10, 15) -- (10, 14);
\draw (13,14) -- (14, 15);
\draw (14, 14) -- (13, 15);
\draw[->] (14.5, 14.5) -- (15.5, 14.5);
\draw (16,14) -- (20, 14) -- (20, 15) -- (18,15) -- (18, 16) -- (17, 16) -- (17, 15) -- (16, 15) -- (16, 14);
\draw (18,14) -- (18, 15);

\draw (0,19) -- (1,19) -- (1,18) -- (2,18) -- (2,19) -- (3,19) -- (3,20) -- (2,20) -- (2,21) -- (1,21) -- (1,20) -- (0,20) -- (0,19);
\draw (0,19) -- (1,18);
\draw (1,19) -- (0,18);
\draw[->] (3.5,19.5) -- (4.5,19.5);
\draw (5,19) -- (5,18) -- (7,18) -- (7,19) -- (8,19) -- (8,20) -- (7,20) -- (7,21) -- (6,21) -- (6,20) -- (5,20) -- (5,19);
\draw (5,19) -- (7,19);

\draw (10,19) -- (11,19) -- (11,18) -- (12,18) -- (12,19) -- (13,19) -- (13,20) -- (12,20) -- (12,21) -- (11,21) -- (11,20) -- (10,20) -- (10,19);
\draw (13,19) -- (14,20);
\draw (14,19) -- (13,20);
\draw[->] (14.5,19.5) -- (15.5,19.5);
\draw (16,19) -- (17,19) -- (17,18) -- (18,18) -- (18,19) -- (20,19) -- (20,20) -- (18,20) -- (18,21) -- (17,21) -- (17,20) -- (16,20) -- (16,19);
\draw (18,19) -- (18,20);

\end{tikzpicture} \\
\caption{Proof of Lemma~\ref{lemma:addsquare}}
\label{fig:proof2}
\end{figure}
\end{proof}

\begin{Defn}
Let $S$ be a fountain set. Let $B_S$ be the smallest subset of $S$ such that for every tile $b \in B_S$, no tile $t \in S \setminus b$ can fit inside of tile $b$. Call $B_S$ the set of \emph{generating tiles} of $S$. 
\end{Defn}

For example, in the tile set $S_2$, the domino is the only generating tile, because it can fit inside of every other tile in $S_2$.

\begin{thm}\label{thm:5tiles}
Let $S$ be a fountain set in dimension $d$. Let $B_S$ be the set of generating tiles of $S$. Let $\mathcal{R}$ be the set of all connected polyominoes that can fit at least one copy of any tile $b \in B_S$. 

The set $S$ can tile any polyomino $R \in \mathcal{R}$. Furthermore, there exists a simple algorithm to generate a tiling of $R$ with $S$ when one exists. 
\end{thm}

For example, for the set $S_2$ for which the domino is the single generating tile, the set of connected polyominoes that can fit at least one copy of the domino is all polyominoes of size at least 2. Thus according to Theorem~\ref{thm:5tiles}, set $S_2$ can tile any polyomino of size at least 2. 

\begin{proof}[Proof of Theorem~\ref{thm:5tiles}]
The proof works by induction on the size of the polyomino $R$ which can fit generating tile $b \in S$. Let $R_k$ be a connected subregion of $R$ such that $|R_k| = k$. \\

\textbf{If $k=|b|$ (base case)}. Since $R_{|b|}$ must be able to fit one copy of a tile $b \in T$, then it is exactly tiled by $b$, which is a tile in $S$. \\

\textbf{If $|b|<k\leq|R|$ (general case)}. Assume the induction hypothesis, that $R_{k-1}$ is tileable with $S$.

Since $R_{k-1}$ is a subregion of $R$ and $R$ is connected, there exists a unit cube $u$ such that $u \in R$, but $u \notin R_{k-1}$, and $u$ is adjacent to another unit cube $v \in R_{k-1} $. Since $R_{k-1}$ is tileable, there exists a valid tiling of $R_{k-1}$, and in this tiling, unit cube $v$ is tiled by a certain tile $s \in S$.  Since $S$ is a fountain set, we can retile the polyomino made up of $s$ and $x$ with $S$. This gives us a valid tiling for of $R_k$.\\

This proof also yields a constructive algorithm to find a tiling of $R$ if one exists, by starting with a single $b$ tile placed in $R$, and then adding one unit cube at a time, retiling locally at each step. Call this algorithm ALG2.  
\end{proof}

\begin{cor}
The set $S_2$ can tile any polyomino of size at least 2.\\ 
\end{cor}


\subsection{Properties of Fountain Sets}

\subsubsection*{Generating Fountain Sets}

To generate a fountain set, we can choose one or multiple generating tiles to form the set $B_S$. Aside from the requirement that they cannot fit inside of each other, the generating tiles can be any polyomino. The following simple procedure completes the fountain set $S$:

\begin{algorithm}
\caption{FSGEN}
\begin{algorithmic}[1]
\State \textbf{Input}: A set of generating tiles $B_S$. 
\State $S \leftarrow B_S$
\While {there exists a tile $b \in S$ such that the region composed of tile $b$ with an adjacent unit cube $u$ is not tileable by $S$}
      \State let tile $t$ be $b$ with the unit cube $u$.
        \State $S \leftarrow S \cup t$.
  \EndWhile
\end{algorithmic}
\end{algorithm}

Notice that FSGEN works not only for square grids in any dimension, but also on triangular or hexagonal grids. 

\begin{lem}\label{fsgen}
Any set $S$ generated by FSGEN is a minimal fountain set.
\end{lem}

\begin{proof}
We first prove that any $S$ generated by FSGEN is a fountain set. FSGEN proceeds until for every tile $s$ in $S$, any tile composed of $s$ plus any adjacent unit cube is tileable by $S$. Thus by definition, once FSGEN terminates, $S$ is a fountain set. 

We now prove that $S$ is a minimal fountain set. In the algorithm FSGEN, every tile $t$ is added because it is a combination of a tile in $S$ and an adjacent unit cube that cannot be tiled by $S$ (see lines 6 and 7 of Algorithm 2). Thus if $t$ was never added to $S$, the fountain property would be broken. Therefore, the final $S$ generated by FSGEN is minimal. 
\end{proof}

\subsubsection*{Size of Fountain Sets}

Another way to look at set $S_2$ from Figure~\ref{fig:5tiles} is to see that it is composed of all possible connected subsets of the plus-tile up to rotation (call these \emph{subtiles}), excluding the unit square. The 2-dimensional plus-tile has one ``central" square and four ``spokes". Extending this idea, a generalized plus-tile in $d$ dimensions has one ``central" cube, and $2d$ ``spokes." Each of its subtiles will have this central cube, and some subset of the spokes. 

\begin{lem}\label{same}
The set composed of all subtiles of the $d$-dimensional plus-tile excluding the unit cube is the set $S_d$ generated by FSGEN using the $d$-dimensional domino as its generating tile. 
\end{lem}

Where the $d$-dimensional domino is the tile composed of two adjacent $d$-dimensional unit cubes. 

\begin{proof}
Let $A$ be the set composed of all the subtiles of the $d$-dimensional plus-tile, excluding the unit cube. Let $B$ be the set generated by FSGEN using the $d$-dimensional domino as its generating tile. 

Since the plus-tile cannot fit two copies of the domino tile, neither can any of its subtiles. Therefore, no tile in $A$ can fit two copies of the domino tile. Furthermore, adding any single adjacent cube to the plus-tile allows it to fit 2 copies of the domino. Therefore, $A$ is the set of all tiles excluding the unit cube that cannot fit two copies of the domino tile. 

On the other hand, any polyomino that cannot fit two dominoes also cannot fit two supertiles of the domino. Therefore, since $B$ is a fountain set and therefore can tile any polyomino of size at least 2, it includes all tiles that cannot fit two copies of the domino tile. Furthermore, $B$ will not include any other tiles, because they are tileable by two dominoes, so in line 6 of Algorithm 2 they will not pass the if statement. 

Therefore, sets $A$ and $B$ are the same set. 
\end{proof}

\begin{cor}
The set $S_d$ composed of the  generalized $d$-dimensional plus-tile and all of its subtiles excluding the unit cube is a fountain set.
\end{cor}

To calculate the size of $S_d$, we will determine how many subtiles of the $d$-dimensional plus-tile there are. First we will count how many different subtiles with exactly $k$ spokes there are, and then sum over the number of spokes possible. 

\begin{lem} 
\label{lem:spokes}
There are $\lfloor \frac{k}{2}\rfloor +1$ different subtiles of the plus-tile with exactly $k$ spokes in dimension $d \ge k$.
\end{lem}

\begin{proof}
 The $k$ spokes can be either: all along different axes, or along $k-1$ axes (with two spokes along the same one), or along $k-2$ axes, etc, or along $\lceil \frac{k}{2} \rceil$ axes. This gives us $k - \lceil \frac{k}{2} \rceil + 1 = \lfloor \frac{k}{2} \rfloor +1$ distinct configurations for a tile with $k$ spokes in dimension $d \geq k$.  \\
\end{proof}

In dimension $d$, there are subtiles of the plus-tile with 1 or 2 or 3 or ... or $2d$ spokes. For tiles with $1,2,...,d$ spokes, Lemma~\ref{lem:spokes} gives the number of possible configurations. For tiles with $d+1, d+2, ..., 2d$ spokes, by symmetry of the empty spoke locations, a tile with $2d-x$ spokes will have the same number of possible configurations as a tile with $x$ spokes, for $x < d$. Also, there is exactly one configuration for the tile with $2d$ spokes. \\

Summing over all the possibilities to get the total number of subtiles of the plus-tile, to get the size of $S_d$:
\[|S_d| = \sum_{k=1}^{d} \left( \left\lfloor \frac k2 \right\rfloor +1 \right) + \sum_{k=d+1}^{2d-1} \left( \left\lfloor \frac{2d-k}{2}\right\rfloor +1 \right) + 1\]
\[|S_d| = \sum_{k=1}^{d} \left( \left\lfloor \frac k2 \right\rfloor +1 \right) + \sum_{j=1}^{d-1} \left( \left\lfloor \frac{j}{2}\right\rfloor +1 \right) + 1\]
\[|S_d| = 2 \cdot \sum_{k=1}^{d-1} \left( \left\lfloor \frac k2 \right\rfloor +1 \right) + \left\lfloor \frac d 2\right\rfloor + 1 + 1 = \frac{d(d+3)}{2}\]
completing the proof of Theorem~\ref{thm:ndimensions}.


\subsection{An Application of Fountain Sets}
\label{sec:applications}

We consider one further classical tiling problem, tiling with the set $S_a$, composed of the domino, L-tile, 3-bar, and T-tile. The algorithm to tile with $S_a$ begins by tiling with $S_2$ using ALG2, then retiling the subregions containing plus-tiles.

\begin{thm}
\label{thm:setta}
Let $R$ be a connected polyomino in the plane $\mathbb{Z}^{2}$, and let $n=|R|$ be the size of $R$. There exists an algorithm that decides in polynomial time whether $R$ can be tiled with the set $S_a$ composed of the domino, L-tile, 3-bar, and T-tile.
\end{thm}

\begin{proof}
We will prove that the only regions not tileable by $S_a$ are crenellated type regions such as the examples shown in Figure~\ref{crenellated}, call the set of these regions $\mathcal{C}$. Any $R \in \mathcal{C}$ is simply connected (no holes) and composed of only overlapped plus shapes. 

\begin{figure}[h]
\begin{tikzpicture}[scale=0.5]
\draw (0,0) -- (0,1) -- (1,1) -- (1,2) -- (2,2)-- (2,1) -- (3,1) -- (3,2) -- (4,2) -- (4,1) -- (5,1);
\draw (0,0) -- (1,0) -- (1,-1) -- (2,-1) -- (2,0) -- (3,0) -- (3,-1) -- (4,-1) -- (4,0) -- (5,0);

\draw[thick,dashed] (5,1) -- (5,2) -- (6,2) -- (6,1) -- (7,1);
\draw[thick,dashed] (5,0) -- (5,-1) -- (6,-1) -- (6,0)-- (7,0);

\draw (7,1) -- (7,2) -- (8,2) -- (8,1) -- (9,1);
\draw (7,0) -- (7,-1) -- (8,-1) -- (8,0) -- (9,0) -- (9,1);

\end{tikzpicture}
\begin{tikzpicture}[scale=0.5]
\draw (0,0) -- (0,1) -- (1,1) -- (1,2) -- (2,2)-- (2,1) -- (3,1) -- (3,2) -- (4,2) -- (4,1) -- (5,1) -- (5,0) -- (4,0) -- (4,-1);
\draw (0,0) -- (1,0) -- (1,-1) -- (2,-1) -- (2,0) -- (3,0) -- (3,-1) -- (2,-1) -- (2,-2) -- (3,-2) -- (3,-3) -- (4,-3) -- (4,-2) -- (5,-2) -- (5,-1) -- (4,-1);
\end{tikzpicture}
\caption{Examples of regions in $\mathcal{C}$.}
\label{crenellated}
\end{figure}
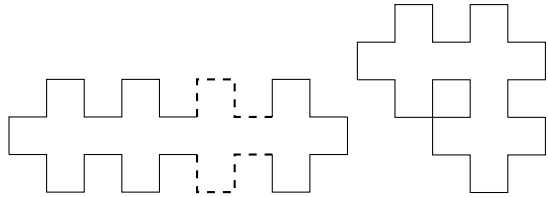

First we show that any $R \in \mathcal{C}$ cannot be tiled by $S_a$. Since $R$ is simply connected, it contains at least one plus-shaped subregion that is only connected to one side. Given the options of tiles in $S_a$, this subregion must be tiled with a T-tile, or else at least one square would be left isolated. Once this T-tile is placed, the remaining region is another $R' \in \mathcal{C}$. Thus, repeating the same logic, another T-tile has to be placed. This continues until the region is almost empty, at which point the only remaining untiled area is a single plus shape, which cannot be tiled by $S_a$. 

Now we show that any region $Q \notin \mathcal{C}$ of size greater than 1 can be tiled by $S_a$. Using ALG2, the region $Q$ can be tiled with $S_2$. Now consider every occurrence of the plus-tile in this tiling of $Q$. Since $Q$ is connected, each plus-tile has at least one adjacent square which is tiled by another tile, and thus can be retiled using the rules from Figure~\ref{fig:proof2}. In this step, since the goal is to remove all plus-tiles if possible, if a plus-tile is adjacent to multiple other tiles, the retiling should be chosen such as not to add a new plus-tile if possible. If not possible, then there is a plus-tile that is only connected to a T-tile at its central tile, and thus the retiling necessarily introduces a new plus-tile. 

However, since $Q \notin \mathcal{C}$, this new plus-tile must be connected to another tile. If this other tile is again only a T-tile connected at its central tile, again the retiling introduces another plus-tile. Again, because $Q \notin \mathcal{C}$, eventually a newly retiled plus-tile will be adjacent to a tile other than a T-tile at its central tile, and then that shape can be retiled without a plus-tile. 

Thus, for any given polyomino $R$, it suffices to check whether $R \in \mathcal{C}$ to determine whether it is tileable by $S_a$, and checking whether $R \in \mathcal{C}$ is an $O(n)$ process.
\end{proof}


\section{Discussion}

It is probably possible to use the fountain set $S_2$ to solve the classic tiling question for the tile set composed of the domino, 3-bar, and L-tile, similarly to how we solved it for set $S_a$ in Section~\ref{sec:applications}. Furthermore, it might then be possible to find how to remove the 3-bars as well, thus answering the classic tiling question for the tile set composed of only the domino and L-tile. If these questions are indeed answerable using this method, it may motivate further study of fountain sets and how they might systematically be used to answer tiling questions. 

In this paper we briefly studied the size and characterization of small fountain sets; it might be interesting in the future to consider more complicated fountain sets and attempt to generalize results. In particular, it might be interesting to search for more algebraic interpretations of fountain sets, perhaps relating them to other algebraic approaches to tiling problems. 

\section*{Acknowledgements}
The authors are thankful to Richard Kenyon for suggesting the initial problem, and to Igor Pak for suggesting references and helpful literature.

\newpage{}

\bibliographystyle{plain}

\bibliography{bibliography}

\end{document}